\documentclass[12pt]{amsart}

\usepackage{graphicx}
\usepackage{subfigure}
\usepackage{eepic}
\usepackage{xcolor}
\selectcolormodel{gray}
\usepackage{tikz}
\usepackage{amsmath}
\usepackage{a4wide}
\usepackage[utf8]{inputenc}
\usepackage{amssymb}
\usepackage{amsopn}
\usepackage{epsfig}
\usepackage{amsfonts}
\usepackage{latexsym}
\usepackage{amsthm}
\usepackage{enumerate}
\usepackage[UKenglish]{babel}
\usepackage{verbatim}
\usepackage{color}

\allowdisplaybreaks[1]

\newtheorem{theorem}{Theorem}[section]
\newtheorem{lemma}[theorem]{Lemma}
\newtheorem{proposition}[theorem]{Proposition}

\theoremstyle{definition}

\newtheorem{question}[theorem]{Question}
\theoremstyle{remark}

\numberwithin{equation}{section}

\newcommand{\R}{\ensuremath{\mathbb{R}}}

\newcommand{\N}{\ensuremath{\mathbb{N}}}

\renewcommand{\O}{\mathcal{O}}

\renewcommand{\L}{\mathcal{L}}%=================== End ======================%

\begin{document}

\title[Equidistribution results for sequences of polynomials]{Equidistribution results for sequences of polynomials}

\author{Simon Baker}
\address{Simon Baker:  School of Mathematics, University of Birmingham, Birmingham, B15 2TT, UK}
\email{simonbaker412@gmail.com}

\date{\today}

\subjclass[2010]{11K06}

\begin{abstract}
Let $(f_n)_{n=1}^{\infty}$ be a sequence of polynomials and $\alpha>1$. In this paper we study the distribution of the sequence  $(f_n(\alpha))_{n=1}^{\infty}$ modulo one. We give sufficient conditions for a sequence $(f_n)_{n=1}^{\infty}$ to ensure that for Lebesgue almost every $\alpha>1$ the sequence $(f_n(\alpha))_{n=1}^{\infty}$ has Poissonian pair correlations. In particular, this result implies that for Lebesgue almost every $\alpha>1$, for any $k\geq 2$ the sequence $(\alpha^{n^k})_{n=1}^{\infty}$ has Poissonian pair correlations. 

%We also give sufficient conditions for a sequence $(f_n)_{n=1}^{\infty}$ to ensure that an analogue of a well known theorem due to Khintchine from Diophantine approximation holds. Importantly our analogue provides an asymptotic for the number of solutions and applies without any monotonicity assumptions. As a consequence of this result we obtain the following statement: Suppose $(f_n)_{n=1}^{\infty}$ is a sequence of polynomials such that the degrees form a strictly increasing sequence, the coefficients are positive and uniformly bounded from above, and the leading coefficients are uniformly bounded away from zero, then for Lebesgue almost every $\alpha>1,$ for any $\delta>0$ we have
%$$\#\left\{1\leq n\leq N:\|f_n(\alpha)\|\leq \frac{1}{n}\right\}=2\log N+\O\left(\log^{1/2}N\log^{3/2+\delta} \log N\right).$$ 

%This significantly generalises and strengthens a result due to Koksma on the sequence $(\alpha^n)$.
\end{abstract}
   
\keywords{Uniform distribution, Poissonian pair correlations.}
\maketitle

\section{Introduction}\label{sec:1}
Given a sequence of real numbers of some number theoretic or dynamical origin, describing its distribution modulo one is a classical problem (see for example \cite{Bug1, Bug, KuiNi} for more on this topic). One approach for describing the distribution of a sequence modulo one is to ask whether it is uniformly distributed. In what follows we let $\{\cdot\}$ denote the fractional part of a real number and $\|\cdot\|$ denote the distance to the nearest integer. We say that a sequence $(x_n)_{n=1}^{\infty}$ is uniformly distributed modulo one if for every pair of real numbers $u,v $ with $0\leq u <v\leq 1$ we have $$\lim_{N\to\infty}\frac{\#\{1\leq n \leq N:\{x_n\}\in [u,v]\}}{N}=v-u.$$ In recent years there has been much interest in a new approach for describing the finer distributional properties of a sequence modulo one. We say that a sequence $(x_n)_{n=1}^{\infty}$ has Poissonian pair correlations if for all $s>0$ we have $$\lim_{N\to\infty}\frac{\#\{1\leq m\neq n\leq N: \|x_n-x_m\|\leq \frac{s}{N}\}}{N}=2s.$$ The original motivation for investigating whether a sequence has Poissonian pair correlations comes from a connection with quantum physics. For certain quantum systems the discrete energy spectra has the form $(\{a_n\alpha\})_{n=1}^{\infty}$ where $\alpha$ is a constant and $(a_n)_{n=1}^{\infty}$ is a sequence of integers. The Berry-Tabor conjecture states that the discrete energy spectrum has Poissonian pair correlations except for in certain degenerate cases. This connection inspired several important contributions due to Rudnick, Sarnak, and Zaharescu, see \cite{RS,RSZ,RZ}. We refer the reader to \cite{AAL} and the references therein for more on this connection between quantum physics and the Poissonian pair correlation property. 

%Another approach for describing the distribution of a sequence modulo one is to consider sequences of intervals $(I_n)_{n=1}^{\infty}$ contained within $[0,1)$, often such that the diameters of the $I_n$ converge to zero, and to ask whether $\{x_n\}\in I_n$ for infinitely many $n$. This approach is often called the shrinking target problem. Whether $\{x_n\}\in I_n$ for infinitely many $n$ is often determined by the convergence/divergence of naturally occurring volume sums. 

Much of the recent interest surrounding whether a sequence $(a_n\alpha)_{n=1}^{\infty}$ has Poissonian pair correlations comes from a connection with additive combinatorics, and more specifically with the so called additive energy of a sequence $(a_n)_{n=1}^{\infty}$. This connection was initially observed by Aistleitner et al in \cite{ALL} and subsequently pursued by several authors. For more on this connection we refer the reader to the survey of Larcher and Stockinger \cite{LS} and the references therein. We remark that the sequence $(n\alpha)_{n=1}^{\infty}$ does not have Poissonian pair correlations for any $\alpha\in \R$. This fact can be seen as a consequence of the three gap theorem. For a short proof of this fact we refer the reader to the aforementioned survey of Larcher and Stockinger \cite{LS}.
%A particularly well studied sequence of real numbers is $(n\alpha)_{n=1}^{\infty}$ for $\alpha\in\R$. It is a well known result that $(n\alpha)_{n=1}^{\infty}$ is uniformly distributed if and only if $\alpha$ is irrational. Another well known result, due to Khintchine (\cite{Khit}), states that if $\Psi:\mathbb{N}\to (0,\infty)$ is a monotone decreasing function such that $\sum_{n=1}^{\infty}\Psi(n)=\infty,$ then Lebesgue almost every $\alpha\in \R$ is contained in the set $$\left\{\alpha\in \R:\|n\alpha\|\leq \Psi(n)\textrm{ for infinitely many }n\in\N\right\}.$$ Importantly, by an example of Duffin and Schaeffer \cite{DufSch}, it is known that the monotonicity assumption on $\Psi$ cannot be removed. 

An interesting family of sequences is obtained by considering $(\alpha^n)_{n=1}^{\infty}$ for $\alpha>1.$ The main source of motivation behind the present work is a desire to obtain a thorough description of the distribution of these sequences. More generally, we are interested in taking a sequence of polynomials $(f_n)_{n=1}^{\infty},$ a real number $\alpha>1,$ and studying the distribution of the sequence $(f_n(\alpha))_{n=1}^{\infty}$ modulo one. 

The study of the distributional properties of $(\alpha^n)_{n=1}^{\infty}$ modulo one dates back to work of Hardy. In \cite{Hardy} he proved that if $\alpha$ is an algebraic number and $\lim_{n\to\infty}\|\alpha^n\|= 0$ then $\alpha$ is a Pisot number. This result was later obtained independently by Pisot in \cite{Pisot2}. Recall that we say a real number $\alpha>1$ is a Pisot number if it is an algebraic integer whose Galois conjugates all have modulus strictly less than one. Pisot had previously shown in \cite{Pisot} that there are at most countably many $\alpha>1$ satisfying $\lim_{n\to\infty}\|\alpha^n\|=0.$ It is an long-standing open question to determine whether there exist any transcendental numbers satisfying $\lim_{n\to\infty}\|\alpha^n\|=0.$ The main result of this paper builds upon the following theorem due to Koksma. 

%These results address the shrinking target problem for $(\alpha^n)_{n=1}^{\infty},$ and the question whether $(\alpha^n)_{n=1}^{\infty}$ is typically uniformly distributed.

\begin{theorem}[\cite{Koks}]
	\label{Koksma uniformly distributed}
For Lebesgue almost every $\alpha>1$ the sequence $(\alpha^n)_{n=1}^{\infty}$ is uniformly distributed modulo one.
\end{theorem}
For some recent results on the distribution of the sequence $(\alpha^n)_{n=1}^{\infty}$ we refer the reader to \cite{A,Bak,BLR, BugMos,Dub,Kah}, \cite[Chapters 2 and 3]{Bug}, and the references therein.

In \cite{ALP} it was shown that if a sequence has Poissonian pair correlations then it is uniformly distributed modulo one. Observe that by our earlier remarks regarding the sequence $(n\alpha)_{n=1}^{\infty},$ and the well known fact that $(n\alpha)_{n=1}^{\infty}$ is uniformly distributed if $\alpha$ is irrational, it follows that having Poissonian pair correlations is a stronger property than being uniformly distributed. With this observation and Theorem \ref{Koksma uniformly distributed} in mind, the following question naturally arises.

\begin{question}
	\label{question}
	Is it true that for Lebesgue almost every $\alpha>1$ the sequence $(\alpha^n)_{n=1}^{\infty}$ has Poissonian pair correlations?
\end{question}
In this paper we do not answer this question. It is worth mentioning that after completion of this paper, the author and Christoph Aistleitner were able to answer Question \ref{question} in the affirmative, see \cite{AB}. The arguments used in this paper rely on a second moment method and good estimates on the Lebesgue measure of certain sets. The arguments used in \cite{AB} combine some of the techniques introduced in this paper with a martingale approach and techniques from Fourier Analysis. The results from \cite{AB} do not imply either Theorem \ref{Main thm} or Theorem \ref{example}, which are the main results of this paper. 
	
	The main result of this paper is the following general theorem which gives sufficient conditions for a sequence of polynomials $(f_n)_{n=1}^{\infty}$ to ensure that for Lebesgue almost every $\alpha>1$ the sequence $(f_n(\alpha))_{n=1}^{\infty}$ has Poissonian pair correlations. 

\begin{theorem}
\label{Main thm}
Suppose $(f_{n})_{n=1}^{\infty}$ is a sequence of polynomials satisfying the following properties:
\begin{enumerate}
	\item The sequence $(\deg(f_n))_{n=1}^{\infty}$ is strictly increasing.
	\item For any $n_2>n_1$ the function $f_{n_2}-f_{n_1}:(1,\infty)\to\mathbb{R}$ is strictly increasing and convex. 
	%\item For any $n\in\N$ we have $f_{n}(\alpha)\geq \alpha^{d_n}$ for all $\alpha>1$.
	\item For any $[a,b]\subset (1,\infty)$, there exists $c_{a,b}>0$ such that for any $\alpha\in[a,b]$ and $n_2>n_1$ we have $$(f_{n_2}-f_{n_1})'(\alpha)\geq c_{a,b}\deg(f_{n_2})\alpha^{\deg(f_{n_2})}.$$

	\item For any $[a,b]\subset (1,\infty)$, there exists $C_{a,b}>1$ such that for any $\alpha\in[a,b]$ and $n_2>n_1$ we have $$\frac{\alpha^{\deg(f_{n_2})}}{C_{a,b}}\leq (f_{n_2}-f_{n_1})(\alpha)\leq C_{a,b}\alpha^{\deg(f_{n_2})}.$$
	\item For any $[a,b]\subset(1,\infty)$ and $C_{a,b}$ as in $(4)$, for $n_1$ sufficiently large the following inequality is satisfied for all $n_2>n_1$
$$\left( \frac{2\deg(f_{n_2})}{\deg(f_{n_1})}-1\right)\log C_{a,b} +\left(\deg(f_{n_1})-\deg(f_{n_2})\right)\log a -\log\left(\deg(f_{n_2})\left(\frac{\deg(f_{n_2})}{\deg(f_{n_1})}-1\right)\right)\leq -3\log n_2.$$
	%$$\left(\deg(f_{n_1})-C-\frac{\deg(f_{n_2})\deg(f_{n_1})}{\deg(f_{n_1})+C}+\frac{\deg(f_{n_2})C}{\deg(f_{n_1})+C}\right)\log a-\log\left(\frac{\deg(f_{n_2})}{\deg(f_{n_1})+C}-1\right)\leq -3\log n_2.$$\textbf{SHOULD $C$ BE $C_{a,b}$ AND IS THIS CONDITION CORRECTLY STATED}
	
\end{enumerate}
Then for Lebesgue almost every $\alpha>1$ the sequence $(f_{n}(\alpha))_{n=1}^{\infty}$ has Poissonian pair correlations. 
\end{theorem}
The fifth assumption appearing in Theorem \ref{Main thm} might seem a little unwieldy. Essentially it is a condition on the growth rate of the sequence $(\deg(f_n))_{n=1}^{\infty}$. Note that it is not satisfied by the sequence $(n)_{n=1}^{\infty}$, which is why we cannot provide an affirmative answer to Question \ref{question}. However for many natural choices of sequences it is a straightforward exercise to check that this assumption is satisfied. As an example, whenever $(\deg(f_n))_{n=1}^{\infty}=(n^k)_{n=1}^{\infty}$ for some $k\geq 2$ then this assumption is satisfied. Similarly, if $(\deg(f_n))_{n=1}^{\infty}=(n!)_{n=1}^{\infty}$ then the fifth assumption is satisfied. These observations imply the following theorem which follows from Theorem \ref{Main thm}.

\begin{theorem}
	\label{example}
	For Lebesgue almost every $\alpha>1$ the sequences  $(\alpha^{n^k})_{n=1}^{\infty}$ and  $(\alpha^{n^k}+\alpha^{n^k-1}+\cdots+\alpha +1)_{n=1}^{\infty}$ have Poissonian pair correlations for all $k\geq 2$. Similarly, for Lebesgue almost every $\alpha>1$ the sequence $(\alpha^{n!})_{n=1}^{\infty}$ has Poissonian pair correlations.
\end{theorem}

\noindent \textbf{Notation.} Throughout this paper we make use of the standard big $\O$ notation, i.e. $X=\O(Y)$ if there exists $C>0$ such that $|X|\leq CY$. When we want to emphasise a dependence for the underlying constant $C$ we will include a subscript, i.e. $X=\O_{a}(Y)$ if $|X|\leq C\cdot Y$ for some $C$ that depends upon $a$. Given a sequence of polynomials $(f_n)_{n=1}^{\infty}$ we will use the notation $(d_n)_{n=1}^{\infty}$ to denote its sequence of degrees. The sequence of polynomials we are referring to will be clear from the context. We let $\L(\cdot)$ denote the Lebesgue measure.
\section{Proof of Theorem \ref{Main thm}}
We will repeatedly use the following lemma in our proof of Theorem \ref{Main thm}.  %and our proof of Theorem \ref{Shrinking target theorem}.

\begin{lemma}
	\label{convexity lemma}
	Let $f:[a,b]\to\mathbb{R}$ be a strictly increasing  differentiable convex function. If $I=[c,d]$ or $I=[c,1]\cup[0,d]$ for some $c,d\in[0,1]$, then we have  $$\frac{\L(I)(b-a)}{1+\L(I)}+\O\left(\frac{\L(I)}{f'(a)}\right)\leq \L\left(\alpha\in [a,b]:\{f(\alpha)\}\in I\right)\leq \frac{\L(I)(b-a)}{1-\L(I)}+\O\left(\frac{\L(I)}{f'(a)}\right).$$ Moreover, if the $I$ above is such that $\L(I)=1/m$ for some $m\in\N,$ then the above can be strengthened to $$\L\left(\alpha\in [a,b]:\{f(\alpha)\}\in I\right)=\L(I)(b-a)+ \O\left(\frac{\L(I)}{f'(a)}\right).$$
\end{lemma}

\begin{proof}
We will prove the statement for general $I$ first. By adding a constant to $f$ if necessary, we can assume without loss of generality that $I=[0,c]$ for some $c\in(0,1]$. We start by proving the lower bound. Consider the following collection of intervals: $$[0,c], [c,2c],\ldots, \left[\left(\left\lfloor  \frac{1}{c}\right \rfloor-1\right)c,\left\lfloor  \frac{1}{c}\right \rfloor c\right], \left[\left\lfloor  \frac{1}{c}\right \rfloor c,1\right].$$ These intervals cover $[0,1)$ and there are $\left\lfloor  \frac{1}{c}\right\rfloor+1$ of them. Therefore, there exists an element of this collection, that we will denote by $J$, such that \begin{equation}
\label{J equation}
\L\left(\alpha\in [a,b]: \{f(\alpha)\}\in J\right)\geq \frac{b-a}{\left\lfloor  \frac{1}{c}\right\rfloor+1}\geq \frac{c(b-a)}{1+c}.
\end{equation} Since $f$ is strictly increasing and convex, the following inequality holds for any interval $L\subset [f(a),f(b)]$ and $t\geq 0$:
\begin{equation}
\label{lower boundaa}
\L(\alpha\in [a,b]:f(\alpha)\in L)\geq \L(\alpha\in [a,b]:f(\alpha)\in L+t).
\end{equation}Using that $f$ is strictly increasing and convex, together with the mean value theorem, we have the following bound. For any interval $L\subset\mathbb{R},$ we have 
\begin{equation}
\label{MVT}
\L(\alpha\in [a,b]:f(\alpha)\in L)=\O\left(\frac{\L(L)}{f'(a)}\right).
\end{equation}

Choosing $t\in[0,1)$ such that $J\subseteq [0,c]+t$ we obtain:
\begin{align*}
\L(\alpha\in [a,b]:\{f(\alpha)\}\in [0,c])&=\sum_{M=\lfloor f(a)\rfloor}^{\lfloor f(b)\rfloor} \L(\alpha\in [a,b]:f(\alpha)\in [0,c]+M)\\
&\stackrel{\eqref{MVT}}{=}\sum_{M=\lfloor f(a)\rfloor+1}^{\lfloor f(b)\rfloor-1} \L(\alpha\in [a,b]:f(\alpha)\in [0,c]+M)+\O\left(\frac{c}{f'(a)}\right)\\
&\stackrel{\eqref{lower boundaa}}{\geq}\sum_{M=\lfloor f(a)\rfloor+1}^{\lfloor f(b)\rfloor-1} \L(\alpha\in [a,b]:f(\alpha)\in [0,c]+t+M)+\O\left(\frac{c}{f'(a)}\right)\\
&\geq \sum_{M=\lfloor f(a)\rfloor+1}^{\lfloor f(b)\rfloor-1} \L(\alpha\in [a,b]:f(\alpha)\in J+M)+\O\left(\frac{c}{f'(a)}\right)\\
&\stackrel{\eqref{MVT}}{=} \sum_{M=\lfloor f(a)\rfloor}^{\lfloor f(b)\rfloor} \L(\alpha\in [a,b]:f(\alpha)\in J+M)+\O\left(\frac{c}{f'(a)}\right)\\
&=\L(\alpha\in [a,b]:\{f(\alpha)\}\in J)+\O\left(\frac{c}{f'(a)}\right)\\
&\stackrel{\eqref{J equation}}{\geq} \frac{c(b-a)}{1+c}+\O\left(\frac{c}{f'(a)}\right).
\end{align*}
This completes the proof of our lower bound. The proof of the upper bound is similar. Note that the upper bound is trivial for $[0,c]$ such that $c\geq 1/2$. As such we restrict our attention to intervals of the form $[0,c]$ for $c<1/2$. This time we consider the collection of intervals $$[0,c], [c,2c],\ldots, \left[\left(\left\lfloor  \frac{1}{c}\right \rfloor-1\right)c,1\right].$$ These intervals cover $[0,1)$ and there are $\left\lfloor  \frac{1}{c}\right \rfloor$ of them. Since the Lebesgue measure of the set of $\alpha$ that are mapped into the intersection of two of these intervals is zero, there exists an element of this collection, that we will denote by $J',$ such that $$\L(\alpha\in [a,b]: \{f(\alpha)\}\in J')\leq \frac{b-a}{\left\lfloor  \frac{1}{c}\right\rfloor}\leq \frac{c(b-a)}{1-c}.$$ Applying \eqref{lower boundaa} in conjunction with \eqref{MVT}, an analogous argument to that given above yields
$$\L(\alpha\in [a,b]:\{f(\alpha)\}\in [0,c])\leq \L(\alpha\in [a,b]: f(\alpha)\in J')+\O\left(\frac{c}{f'(a)}\right).$$ Then by the definition of $J'$ we have $$\L(\alpha\in [a,b]:\{f(\alpha)\}\in [0,c])\leq \frac{c(b-a)}{1-c}+\O\left(\frac{c}{f'(a)}\right).$$ This completes our proof of the upper bound.

%Using the fact $f$ is convex and strictly increasing, it follows that if $L\subset \R$ is an interval and $t\geq 0$ is such that $L-t\in [f(a),f(b)],$ then  
%\begin{equation}
%\label{upper boundaa}
%\L(\alpha\in [a,b]:f(\alpha)\in L)\leq \L(\alpha\in %[a,b]:f(\alpha)\in L-t).
%\end{equation}

To deduce the stronger statement when $\L(I)=1/m$ for some $m\in\N,$ notice that the two collections of intervals appearing in the proof of the lower bound and upper bound can both be replaced by the single collection given by the intervals $$[0,1/m], [1/m,2/m],\ldots, [(m-1)/m,1].$$ Importantly this collection consists of exactly $m$ elements. Repeating the arguments given above for this collection yields the stronger statement.

\end{proof}
The following proposition is the tool that allows us to prove Theorem \ref{Main thm}.
\begin{proposition}
	\label{error prop}
	Let $(f_n)_{n=1}^{\infty}$ be a sequence of polynomials satisfying the hypothesis of Theorem \ref{Main thm}.
	%\begin{itemize}
	%		\item For each $n\in \N$ the function $f_{n}:(1,\infty)\to \mathbb{R}$ is strictly increasing, and convex. Moreover, for any $n_2>n_1$ the function $f_{n_2}-f_{n_1}:(1,\infty)\to\mathbb{R}$ is strictly increasing and convex.
		%\item For any $n\in\N$ we have $f_{n}(\alpha)\geq \alpha^{d_n}$ for all $\alpha>1$.
		%\item For any $a>1$, there exists $c_a>0$ such that for any $\alpha\geq a$ and $n_2>n_1$, we have $$(f_{n_2}-f_{n_1})'(\alpha)\geq c_{\alpha}\alpha^{d_{n_2}}.$$
		%\item For any $a>1$ and any $C\in\mathbb{N}$, for all but finitely many $n_2>n_1$ the following inequality is satisfied 
		%$$\left(d_{n_1}-C-\frac{d_{n_2}d_{n_1}}{d_{n_1}+C}+\frac{d_{n_2}C}{d_{n_1}+C}\right)\log a-\log\left(\frac{d_{n_2}}{d_{n_1}+C}-1\right)\leq -3\log n_2.$$
		%\item For any $a>1$, there exists $C_{a}\in\mathbb{N}$ such that for all $\alpha\geq a$ and $n_2>n_1,$ we have $$\alpha^{d_{n_2}+C_{a}}\geq f_{n_2}(\alpha)-f_{n_1}(\alpha)\geq \alpha^{d_{n_2}-C_{a}}.$$
	%\end{itemize}
	Then for any $[a,b]\subset(1,\infty)$ and $s>0$ we have $$\int_{a}^{b}\left(\frac{\#\{1\leq m\neq n\leq N: \|f_{n}(\alpha)-f_{m}(\alpha)\|\leq \frac{s}{N}\}}{N}-2s\right)^2 d\alpha=\O_{s,a,b}\left(\frac{1}{N}\right).$$
\end{proposition}
We split our proof of Proposition \ref{error prop} into a series of lemmas. Throughout this paper $\chi_{A}$ denotes the indicator function on a set $A\subset \mathbb{R}$. We start by expanding the bracket appearing within the integral to obtain: 
\begin{align}
\label{split}
&\int_{a}^{b}\left(\frac{\#\{1\leq m\neq n\leq N: \|f_{n}(\alpha)-f_{m}(\alpha)\|\leq \frac{s}{N}\}}{N}-2s\right)^2 \, d\alpha\\
=&\frac{1}{N^2}\sum_{\substack{1\leq m\neq n\leq N\\ 1\leq p\neq q\leq N\\(m,n)\neq (p,q)}}\int_{a}^{b}\chi_{[0,\frac{s}{N}]}(\|f_{n}(\alpha)-f_{m}(\alpha)\|)\chi_{[0,\frac{s}{N}]}(\|f_{q}(\alpha)-f_{p}(\alpha)\|) \,d\alpha\nonumber \\
+&\frac{1}{N^2}\sum_{1\leq m\neq n\leq N}\int_{a}^{b}\chi_{[0,\frac{s}{N}]}(\|f_{n}(\alpha)-f_{m}(\alpha)\|)\,d\alpha\nonumber\\
-&\frac{4s}{N}\sum_{1\leq m\neq n\leq N}\int_{a}^{b}\chi_{[0,\frac{s}{N}]}(\|f_{n}(\alpha)-f_{m}(\alpha)\|)\,d\alpha\nonumber\\
+ &4s^2(b-a).\nonumber
\end{align} We will focus on each term on the right hand side of \eqref{split} individually. It is useful at this point to rewrite the first term as follows:
\begin{align}
\label{double split}
&\frac{1}{N^2}\sum_{\substack{1\leq m\neq n\leq N\\ 1\leq p\neq q\leq N\\(m,n)\neq (p,q)}}\int_{a}^{b}\chi_{[0,\frac{s}{N}]}(\|f_{n}(\alpha)-f_{m}(\alpha)\|)\chi_{[0,\frac{s}{N}]}(\|f_{q}(\alpha)-f_{p}(\alpha)\|)\, d\alpha\nonumber\\
=&\frac{4}{N^2}\sum_{\substack{1\leq m< n\leq N\\ 1\leq p< q\leq N\\(m,n)\neq (p,q)}}\int_{a}^{b}\chi_{[0,\frac{s}{N}]}(\|f_{n}(\alpha)-f_{m}(\alpha)\|)\chi_{[0,\frac{s}{N}]}(\|f_{q}(\alpha)-f_{p}(\alpha)\|)\, d\alpha\nonumber\\
=&\frac{4}{N^2}\sum_{\substack{1\leq m< n\leq N\\ 1\leq p< q\leq N\\ q\neq n}}\int_{a}^{b}\chi_{[0,\frac{s}{N}]}(\|f_{n}(\alpha)-f_{m}(\alpha)\|)\chi_{[0,\frac{s}{N}]}(\|f_{q}(\alpha)-f_{p}(\alpha)\|)\, d\alpha\nonumber\\
+&\frac{4}{N^2}\sum_{1\leq m\neq p< n\leq N}\int_{a}^{b}\chi_{[0,\frac{s}{N}]}(\|f_{n}(\alpha)-f_{m}(\alpha)\|)\chi_{[0,\frac{s}{N}]}(\|f_{n}(\alpha)-f_{p}(\alpha)\|)\, d\alpha\nonumber\\
=&\frac{8}{N^2}\sum_{\substack{1\leq m< n\leq N\\ 1\leq p< q\leq N\\ q< n}}\int_{a}^{b}\chi_{[0,\frac{s}{N}]}(\|f_{n}(\alpha)-f_{m}(\alpha)\|)\chi_{[0,\frac{s}{N}]}(\|f_{q}(\alpha)-f_{p}(\alpha)\|)\, d\alpha\\
+&\frac{8}{N^2}\sum_{\substack{1\leq m<p< n\leq N}}\int_{a}^{b}\chi_{[0,\frac{s}{N}]}(\|f_{n}(\alpha)-f_{m}(\alpha)\|)\chi_{[0,\frac{s}{N}]}(\|f_{n}(\alpha)-f_{p}(\alpha)\|)\, d\alpha.\nonumber
\end{align} The behaviour of the two terms on the right hand side of \eqref{double split} is described by the following lemmas.
\begin{lemma}
	\label{lemma1}
	Suppose $(f_n)_{n=1}^{\infty}$ is a sequence of polynomials satisfying the hypothesis of Theorem \ref{Main thm}. Then for any $[a,b]\subset(1,\infty)$ and $s>0$ we have
	$$\frac{8}{N^2}\sum_{\substack{1\leq m< n\leq N\\ 1\leq p< q\leq N\\ q< n}}\int_{a}^{b}\chi_{[0,\frac{s}{N}]}(\|f_{n}(\alpha)-f_{m}(\alpha)\|)\chi_{[0,\frac{s}{N}]}(\|f_{q}(\alpha)-f_{p}(\alpha)\|)\, d\alpha\leq 4s^2 (b-a)+\O_{s,a,b}\left(\frac{1}{N}\right).$$
\end{lemma}
%\begin{lemma}
%	\label{lemmaa}
%Suppose $(f_n)$ is a sequence of polynomials satisfying the hypothesis of Theorem \ref{Main thm}. Then for any $s>0$ we have 
%$$\frac{1}{N^2}\sum_{\substack{1\leq j\neq k\leq N\\ 1\leq l\neq m\leq N\\(j,k)\neq (l,m)}}\int_{a}^{b}\chi_{[0,\frac{s}{N}]}(\|f_{k}(\alpha)-f_{j}(\alpha)\|)\chi_{[0,\frac{s}{N}]}(\|f_{m}(\alpha)-f_{l}(\alpha)\|)=4s^2(b-a)+\O_{s,a,b}\left(\frac{1}{N}\right)$$
%\end{lemma}
\begin{proof}
%$\lfloor f_{q}(a)-f_{p}(a)\rfloor \leq M \leq \lceil f_{q}(b)-f_{p}(b)\rceil,$
To each $1\leq p<q\leq N$ and $M\in \big[\lfloor f_{q}(a)-f_{p}(a)\rfloor,\lceil f_{q}(b)-f_{p}(b)\rceil\big]$  we associate the interval $$I_{M,q,p}:=\left\{\alpha\in[a,b]:f_{q}(\alpha)-f_{p}(\alpha)\in \left[M-\frac{s}{N},M+\frac{s}{N}\right]\right\}.$$ This is an interval because the function $f_{q}-f_{p}$ is strictly increasing. We note that \begin{equation}
\label{union}\left\{\alpha\in[a,b]:\|f_{q}(\alpha)-f_{p}(\alpha)\|\leq \frac{s}{N}\right\}=\bigcup_{M=\lfloor f_{q}(a)-f_{p}(a)\rfloor}^{\lceil f_{q}(b)-f_{p}(b)\rceil} I_{M,q,p}.
\end{equation} We denote the left hand point of each non-empty $I_{M,q,p}$ by $c_{M,q,p}$.

Note that $\|f_{n}(\alpha)-f_{m}(\alpha)\|\in[0,\frac{s}{N}]$ if any only if $\{f_{n}(\alpha)-f_{m}(\alpha)\} \in[0,\frac{s}{N}]\cup [1-\frac{s}{N},1)$. Therefore by an application of Lemma \ref{convexity lemma} we have
\begin{align}
\label{second split}
&\sum_{\substack{1\leq m< n\leq N\\ 1\leq p< q\leq N\\ q< n}}\int_{a}^{b}\chi_{[0,\frac{s}{N}]}(\|f_{n}(\alpha)-f_{m}(\alpha)\|)\chi_{[0,\frac{s}{N}]}(\|f_{q}(\alpha)-f_{p}(\alpha)\|)\, d\alpha\nonumber\\
&=\sum_{\substack{1\leq m< n\leq N\\ 1\leq p< q\leq N\\ q< n}}\sum_{M=\lfloor f_{q}(a)-f_{p}(a)\rfloor}^{\lceil f_{q}(b)-f_{p}(b)\rceil}\int_{I_{M,q,p}}\chi_{[0,\frac{s}{N}]}(\|f_{n}(\alpha)-f_{m}(\alpha)\|)\, d\alpha\nonumber\\
&\leq \sum_{\substack{1\leq m< n\leq N\\ 1\leq p< q\leq N\\ q< n}}\sum_{M=\lfloor f_{q}(a)-f_{p}(a)\rfloor}^{\lceil f_{q}(b)-f_{p}(b)\rceil}\frac{\L(I_{M,q,p})2s/N}{1-2s/N}\\
&+\O\left(\sum_{\substack{1\leq m< n\leq N\\ 1\leq p< q\leq N\\ q< n}}\sum_{M=\lfloor f_{q}(a)-f_{p}(a)\rfloor}^{\lceil f_{q}(b)-f_{p}(b)\rceil}\frac{2s}{N(f_{n}-f_{m})'(c_{M,q,p})}\right)\nonumber. 
\end{align}
We now treat the two terms appearing in \eqref{second split} separately.\\

\noindent \textbf{Bounding the first term in \eqref{second split}.}

By an application of Lemma \ref{convexity lemma} and \eqref{union} we have
\begin{align}
\label{soon to sub}
&\sum_{\substack{1\leq m< n\leq N\\ 1\leq p< q\leq N\\ q< n}}\sum_{M=\lfloor f_{q}(a)-f_{p}(a)\rfloor}^{\lceil f_{q}(b)-f_{p}(b)\rceil}\frac{\L(I_{M,q,p})2s/N}{1-2s/N}\\
=&\frac{2s}{N-2s}\sum_{\substack{1\leq m< n\leq N\\ 1\leq p< q\leq N\\ q< n}}\sum_{M=\lfloor f_{q}(a)-f_{p}(a)\rfloor}^{\lceil f_{q}(b)-f_{p}(b)\rceil}\L(I_{M,q,p})\nonumber\\
\leq& \frac{2s}{N-2s}\sum_{\substack{1\leq m< n\leq N\\ 1\leq p< q\leq N\\ q< n}}\left(\frac{(b-a)2s/N}{1-2s/N}+\O\left(\frac{2s}{N(f_{q}-f_{p})'(a)}\right)\right)\nonumber\\
=&\frac{4s^2(b-a)}{(N-2s)^2}\sum_{\substack{1\leq m< n\leq N\\ 1\leq p< q\leq N\\ q< n}}1\nonumber\\
 +&\O\left(\frac{4s^2}{N(N-2s)}\sum_{\substack{1\leq m< n\leq N\\ 1\leq p< q\leq N\\ q< n}}\frac{1}{(f_{q}-f_{p})'(a)}\right)\nonumber.
\end{align} By our third assumption we know that $(f_{q}-f_{p})'(a)\geq c_{a,b}d_{q}a^{d_q}$ for $q>p$. We also know by our first assumption that $(d_n)_{n=1}^{\infty}$ is a strictly increasing sequence of natural numbers, therefore $d_n\geq n$ for all $n\in\N.$ This implies $(f_{q}-f_{p})'(a)\geq c_{a,b}qa^{q}$ for $q>p.$ Therefore
\begin{align}
\label{BOUND}
\sum_{\substack{1\leq m< n\leq N\\ 1\leq p< q\leq N\\ q< n}}\frac{1}{(f_{q}-f_{p})'(a)}
=& \sum_{n=3}^{N}\sum_{m=1}^{n-1}\sum_{p=1}^{n-2}\sum_{q=p+1}^{n-1}\frac{1}{(f_{q}-f_{p})'(a)}\nonumber\\
=&\O_{a,b}\left( \sum_{n=3}^{N}\sum_{m=1}^{n-1}\sum_{p=1}^{n-2}\sum_{q=p+1}^{n-1}	\frac{1}{qa^q}\right)\nonumber\\
=& \O_{a,b}\left(\sum_{n=3}^{N}\sum_{m=1}^{n-1}\sum_{p=1}^{n-2}\frac{1}{(p+1)a^{p+1}}\right)\nonumber\\
=& \O_{a,b}\left(\sum_{n=3}^{N}\sum_{m=1}^{n-1}1\right)\nonumber\\
=&\O_{a,b}(N^2).
\end{align}
A straightforward calculation yields
\begin{equation}
\label{count}
\sum_{\substack{1\leq m< n\leq N\\ 1\leq p< q\leq N\\ q< n}}1=\frac{N^4}{8} +\O(N^3).
\end{equation}
Substituting \eqref{BOUND} and \eqref{count} into \eqref{soon to sub}, we see that the following holds for the first term in \eqref{second split}
$$\sum_{\substack{1\leq m< n\leq N\\ 1\leq p< q\leq N\\ q< n}}\sum_{M=\lfloor f_{q}(a)-f_{p}(a)\rfloor}^{\lceil f_{q}(b)-f_{p}(b)\rceil}\frac{\L(I_{M,q,p})2s/N}{1-2s/N}
\leq\frac{s^2(b-a)N^4}{2(N-2s)^2}+\O_{s,a,b}(N).$$ It is easy to show that $$\frac{s^2(b-a)N^4}{2(N-2s)^2}=\frac{s^2(b-a)N^2}{2}+\O_{s,a,b}(N).$$ Therefore the following holds for the first term in \eqref{second split}
\begin{equation}
\label{bound 1}
\sum_{\substack{1\leq m< n\leq N\\ 1\leq p< q\leq N\\ q< n}}\sum_{M=\lfloor f_{q}(a)-f_{p}(a)\rfloor}^{\lceil f_{q}(b)-f_{p}(b)\rceil}\frac{\L(I_{M,q,p})2s/N}{1-2s/N}
\leq \frac{s^2(b-a)N^2}{2}+\O_{s,a,b}(N).
\end{equation}
\noindent \textbf{Bounding the second term in \eqref{second split}.}

By the fifth assumption listed in Theorem \ref{Main thm}, we know that there exists some $N_1\in\mathbb{N}$ for which
\begin{equation}
\label{hypothesis bound}
\left( \frac{2d_n}{d_q}-1\right)\log C_{a,b} +\left(d_q-d_n\right)\log a -\log\left(d_n\left(\frac{d_n}{d_q}-1\right)\right)\leq -3\log n
\end{equation}whenever $q\geq N_1$ and $n>q$. The equation below describes the error that occurs by restricting the second term in \eqref{second split} to $q\geq N_1$. As we will see, this error will be negligible. We have
\begin{align}
\label{N error}
\sum_{\substack{1\leq m< n\leq N\\ 1\leq p< q\leq N\\ q< n}}\sum_{M=\lfloor f_{q}(a)-f_{p}(a)\rfloor}^{\lceil f_{q}(b)-f_{p}(b)\rceil}\frac{2s}{N(f_{n}-f_{m})'(c_{M,q,p})}&=\sum_{\substack{1\leq m< n\leq N\\ 1\leq p< q\leq N\\ N_1\leq q< n}}\sum_{M=\lfloor f_{q}(a)-f_{p}(a)\rfloor}^{\lceil f_{q}(b)-f_{p}(b)\rceil}\frac{2s}{N(f_{n}-f_{m})'(c_{M,q,p})}\nonumber \\
&+\sum_{\substack{1\leq m< n\leq N\\ 1\leq p< q< N_1\\ q< n}}\sum_{M=\lfloor f_{q}(a)-f_{p}(a)\rfloor}^{\lceil f_{q}(b)-f_{p}(b)\rceil}\frac{2s}{N(f_{n}-f_{m})'(c_{M,q,p})}\nonumber\\
&=\sum_{\substack{1\leq m< n\leq N\\ 1\leq p< q\leq N\\ N_1\leq q< n}}\sum_{M=\lfloor f_{q}(a)-f_{p}(a)\rfloor}^{\lceil f_{q}(b)-f_{p}(b)\rceil}\frac{2s}{N(f_{n}-f_{m})'(c_{M,q,p})}\nonumber\\
&+\O_{s,a,b}\left(\sum_{\substack{1\leq m< n\leq N\\ 1\leq p< q< N_1\\ q< n}}\frac{1}{N}\right)\nonumber\\
&=\sum_{\substack{1\leq m< n\leq N\\ 1\leq p< q\leq N\\ N_1\leq q< n}}\sum_{M=\lfloor f_{q}(a)-f_{p}(a)\rfloor}^{\lceil f_{q}(b)-f_{p}(b)\rceil}\frac{2s}{N(f_{n}-f_{m})'(c_{M,q,p})}\\
&+\O_{s,a,b}\left(N\right).\nonumber
\end{align}In the penultimate equality we used that for any $q\leq N_1,$ for $p<q$ and $m,n$ satisfying $m<n$ and $q<n$, we have 
$$\sum_{M=\lfloor f_{q}(a)-f_{p}(a)\rfloor}^{\lceil f_{q}(b)-f_{p}(b)\rceil}\frac{1}{(f_{n}-f_{m})'(c_{M,q,p})}=\O_{a,b}(1).$$

We now bound the first term on the right hand side of \eqref{N error}. Recall that by our fourth assumption there exists $C_{a,b}>1$ such that $f_{q}(\alpha)-f_p(\alpha)\leq C_{a,b}\alpha^{d_q}$ for all $\alpha\in[a,b].$ Therefore $$C_{a,b}c_{M,q,p}^{d_q}\geq M-\frac{s}{N}.$$ Increasing $C_{a,b}$ if necessary, we may assume without loss of generality that $$C_{a,b}c_{M,q,p}^{d_q}\geq M+2$$ holds for all $c_{M,q,p}.$ Therefore
\begin{equation}
\label{growth bound}
c_{M,q,p}\geq \left(\frac{M+2}{C_{a,b}}\right)^{1/d_q}.
\end{equation}
Using the fact that $f_n-f_m$ is convex, we see that \eqref{growth bound} implies 
\begin{equation*}
\label{derivative bound}
(f_n-f_m)'(c_{M,q,p})\geq (f_n-f_m)'\left(\left(\frac{M+2}{C_{a,b}}\right)^{1/d_q}\right).
\end{equation*} Therefore
\begin{align}
\label{lemon}
&\sum_{\substack{1\leq m< n\leq N\\ 1\leq p< q\leq N\\ N_1\leq q< n}}\sum_{M=\lfloor f_{q}(a)-f_{p}(a)\rfloor}^{\lceil f_{q}(b)-f_{p}(b)\rceil}\frac{2s}{N(f_{n}-f_{m})'(c_{M,q,p})}\\
\leq&\sum_{\substack{1\leq m< n\leq N\\ 1\leq p< q\leq N\\ N_1\leq q< n}}\sum_{M=\lfloor f_{q}(a)-f_{p}(a)\rfloor}^{\lceil f_{q}(b)-f_{p}(b)\rceil}\frac{2s}{N(f_n-f_m)'\left(\left(\frac{M+2}{C_{a,b}}\right)^{1/d_q}\right)}.\nonumber
\end{align} We would now like to be able to use our third assumption to assert that
$$(f_n-f_m)'\left(\left(\frac{M+2}{C_{a,b}}\right)^{1/d_q}\right)\geq c_{a,b}d_n\left(\frac{M+2}{C_{a,b}}\right)^{d_n/d_q}.$$ However we cannot apply this assumption directly since $\left(\frac{M+2}{C_{a,b}}\right)^{1/d_q}$ is not necessarily contained in $[a,b]$. However, we know by our fourth assumption that $f_q(a)-f_p(a)\geq \frac{a^{d_q}}{C_{a,b}}$ and $f_q(b)-f_p(b)\leq C_{a,b}b^{d_q}$ for all $p<q$. Using these facts, together with the property $\lim_{q\to\infty}d_q=\infty$, it follows that for $q$ sufficiently large, for $p<q$ and $M\in\big[\lfloor f_{q}(a)-f_{p}(a)\rfloor, \lceil f_{q}(b)-f_{p}(b)\rceil\big],$ we have 
\begin{equation}
\label{M inclusion}
\left(\frac{M+2}{C_{a,b}}\right)^{1/d_q}\in\left[ \frac{1+a}{2}, 2b\right].
\end{equation} Without loss of generality we can assume that the $N_1$ we chose earlier was sufficiently large to guarantee \eqref{M inclusion} holds for any $M\in\big[\lfloor f_{q}(a)-f_{p}(a)\rfloor, \lceil f_{q}(b)-f_{p}(b)\rceil\big]$ for $q\geq N_1$ and $p<q$. In which case we can apply our third assumption for the interval $[\frac{1+a}{2}, 2b]$ to assert that there exists a constant $c_{a,b}'>0$ such that $$(f_n-f_m)'\left(\left(\frac{M+2}{C_{a,b}}\right)^{1/d_q}\right)\geq c_{a,b}'d_n\left(\frac{M+2}{C_{a,b}}\right)^{d_n/d_q}$$ for any $M\in\big[\lfloor f_{q}(a)-f_{p}(a)\rfloor, \lceil f_{q}(b)-f_{p}(b)\rceil\big]$ for $q\geq N_1$ and $p<q$.
Using this bound in \eqref{lemon} we have
\begin{align}
\label{handstand}
&\sum_{\substack{1\leq m< n\leq N\\ 1\leq p< q\leq N\\ N_1\leq q< n}}\sum_{M=\lfloor f_{q}(a)-f_{p}(a)\rfloor}^{\lceil f_{q}(b)-f_{p}(b)\rceil}\frac{2s}{N(f_{n}-f_{m})'(c_{M,q,p})}\nonumber\\
=&\mathcal{O}_{s,a,b}\left(\frac{1}{N} \sum_{\substack{1\leq m< n\leq N\\ 1\leq p< q\leq N\\ N_1\leq q< n}}\sum_{M=\lfloor f_{q}(a)-f_{p}(a)\rfloor}^{\lceil f_{q}(b)-f_{p}(b)\rceil}\frac{C_{a,b}^{d_n/d_q}}{d_n(M+2)^{d_n/d_q}}\right)\nonumber\\
=&\mathcal{O}_{s,a,b}\left(\frac{1}{N} \sum_{\substack{1\leq m< n\leq N\\ 1\leq p< q\leq N\\ N_1\leq q< n}}\frac{C_{a,b}^{d_n/d_q}}{d_n}\int_{\lfloor f_{q}(a)-f_{p}(a)\rfloor-1}^{\lceil f_{q}(b)-f_{p}(b)\rceil-1}\frac{1}{(x+2)^{d_n/d_q}}\,dx\right)\nonumber\\
=&\O_{s,a,b}\left(\frac{1}{N} \sum_{\substack{1\leq m< n\leq N\\ 1\leq p< q\leq N\\ N_1\leq q< n}}\frac{(C_{a,b}^{d_n/d_q}(\lfloor f_{q}(a)-f_{p}(a)\rfloor+1)^{1-d_n/d_q}}{d_n(d_n/d_q -1)}\right)\nonumber\\
=&\O_{s,a,b}\left(\frac{1}{N} \sum_{\substack{1\leq m< n\leq N\\ 1\leq p< q\leq N\\ N_1\leq q< n}}\frac{(C_{a,b}^{d_n/d_q}( f_{q}(a)-f_{p}(a))^{1-d_n/d_q}}{d_n(d_n/d_q -1)}\right)\nonumber\\
=&\O_{s,a,b}\left(\frac{1}{N} \sum_{\substack{1\leq m< n\leq N\\ 1\leq p< q\leq N\\ N_1\leq q< n}}\frac{C_{a,b}^{d_n/d_q}(a^{d_q}/C_{a,b})^{1-d_n/d_q}}{d_n(d_n/d_q -1)}\right).
\end{align}In the final line we used our fourth assumption that $f_{q}(a)-f_p(a)\geq \frac{a^{d_q}}{C_{a,b}}$. By \eqref{hypothesis bound} we know that $$\frac{C_{a,b}^{d_n/d_q}(a^{d_q}/C_{a,b})^{1-d_n/d_q}}{d_n(d_n/d_q -1)}\leq \frac{1}{n^3}$$ whenever $ q\geq N_1$ and $n>q$. Substituting this bound into \eqref{handstand} we obtain
\begin{align*}
\sum_{\substack{1\leq m< n\leq N\\ 1\leq p< q\leq N\\ N_1\leq q< n}}\sum_{M=\lfloor f_{q}(a)-f_{p}(a)\rfloor}^{\lceil f_{q}(b)-f_{p}(b)\rceil}\frac{2s}{N(f_{n}-f_{m})'(c_{M,q,p})}&=\O_{s,a,b}\left(\frac{1}{N} \sum_{\substack{1\leq m< n\leq N\\ 1\leq p< q\leq N\\ N_1\leq q< n}}\frac{1}{n^3}\right)\\
&=\O_{s,a,b}\left(\frac{1}{N}\sum_{q=N_1}^{N-1}\sum_{p=1}^{q-1}\sum_{n=q+1}^{N}\sum_{m=1}^{n-1}\frac{1}{n^3}\right)\\
&=\O_{s,a,b}\left(\frac{1}{N}\sum_{q=N_1}^{N-1}\sum_{p=1}^{q-1}\sum_{n=q+1}^{N}\frac{1}{n^2}\right)\\
&=\O_{s,a,b}\left(\frac{1}{N}\sum_{q=N_1}^{N-1}\sum_{p=1}^{q-1}\frac{1}{q}\right)\\
&=\O_{s,a,b}\left(\frac{1}{N}\sum_{q=N_1}^{N-1}1\right).\\
&=\O_{s,a,b}(1)
\end{align*}

% Choose $a'\in(1,a),$ then since $\lfloor f_{k}(a)-f_{j}(a)\rfloor\geq c\cdot a^{x_k}$ we have the following for $k$ sufficiently large:
%$$\sum_{\substack{1\leq j< k\leq N\\ 1\leq l< m\leq N\\ k< m}}\frac{(\lfloor f_{k}(a)-f_{j}(a)\rfloor-2)^{1-x_m/x_k}}{x_m/x_k -1}\leq \sum_{\substack{1\leq j< k\leq N\\ 1\leq l< m\leq N\\ k< m}}\frac{a'^{x_k-x_m}}{x_m/x_k -1}.$$ By assumptions $$\frac{a'^{x_k-x_m}}{x_m/x_k -1}\leq a''^{-m}$$ for some suitable $a''>1$.
%\begin{align*}
%\sum_{\substack{1\leq j< k\leq N\\ 1\leq l< m\leq N\\ k< m}}a''^{-m}&=\sum_{k=2}^{N-1}\sum_{j=1}^{k-1}\sum_{m=k+1}^{N}\sum_{l=1}^{m-1}a''^{-m}\\
%&=\sum_{k=2}^{N-1}\sum_{j=1}^{k-1}\sum_{m=k+1}^{N}(m-1)a''^{-m}\\
%&=\O(\sum_{k=2}^{N-1}\sum_{j=1}^{k-1}a'''^{-k})\\
%&=\O(\sum_{k=2}^{N-1}(k-1)a'''^{-k})\\
%&=\O(1).
%\end{align*}
\noindent Therefore 
\begin{equation*}
\sum_{\substack{1\leq m< n\leq N\\ 1\leq p< q\leq N\\ N_1\leq q< n}}\sum_{M=\lfloor f_{q}(a)-f_{p}(a)\rfloor}^{\lceil f_{q}(b)-f_{p}(b)\rceil}\frac{2s}{N(f_{n}-f_{m})'(c_{M,q,p})}=\O_{s,a,b}\left(1\right).
\end{equation*}
Which when combined with \eqref{N error} gives
\begin{equation}
\label{bound2}
\sum_{\substack{1\leq m< n\leq N\\ 1\leq p< q\leq N\\ q< n}}\sum_{M=\lfloor f_{q}(a)-f_{p}(a)\rfloor}^{\lceil f_{q}(b)-f_{p}(b)\rceil}\frac{2s}{N(f_{n}-f_{m})'(c_{M,q,p})}=\O_{s,a,b}\left(N\right).
\end{equation}
Substituting \eqref{bound 1} and \eqref{bound2} into \eqref{second split} we obtain the desired inequality:
\begin{equation*}
\frac{8}{N^2}\sum_{\substack{1\leq m< n\leq N\\ 1\leq p< q\leq N\\ q< n}}\int_{a}^{b}\chi_{[0,\frac{s}{N}]}(\|f_{n}(\alpha)-f_{m}(\alpha)\|)\chi_{[0,\frac{s}{N}]}(\|f_{q}(\alpha)-f_{p}(\alpha)\|)\, d\alpha\leq 4s^2 (b-a)+\O_{s,a,b}\left(\frac{1}{N}\right).
\end{equation*}
\end{proof}
\begin{lemma}
	\label{lemma2}
	Suppose $(f_n)_{n=1}^{\infty}$ is a sequence of polynomials satisfying the hypothesis of Theorem \ref{Main thm}. Then for any $[a,b]\subset(1,\infty)$ and $s>0$ we have
	$$\frac{8}{N^2}\sum_{\substack{1\leq m<p< n\leq N}}\int_{a}^{b}\chi_{[0,\frac{s}{N}]}(\|f_{n}(\alpha)-f_{m}(\alpha)\|)\chi_{[0,\frac{s}{N}]}(\|f_{n}(\alpha)-f_{p}(\alpha)\|)\, d\alpha=\O_{s,a,b}\left(\frac{1}{N}\right)$$
\end{lemma}
\begin{proof}
Notice that if $\chi_{[0,\frac{s}{N}]}(\|f_{n}(\alpha)-f_{m}(\alpha)\|)=1$ and $\chi_{[0,\frac{s}{N}]}(\|f_{n}(\alpha)-f_{p}(\alpha)\|)=1$ then $\chi_{[0,\frac{2s}{N}]}(\|f_{p}(\alpha)-f_{m}(\alpha)\|)=1.$ Therefore 
\begin{align}
\label{bound3}&\sum_{\substack{1\leq m<p< n\leq N}}\int_{a}^{b}\chi_{[0,\frac{s}{N}]}(\|f_{n}(\alpha)-f_{m}(\alpha)\|)\chi_{[0,\frac{s}{N}]}(\|f_{n}(\alpha)-f_{p}(\alpha)\|)\, d\alpha\\
\leq &\sum_{\substack{1\leq m<p< n\leq N}}\int_{a}^{b}\chi_{[0,\frac{2s}{N}]}(\|f_{p}(\alpha)-f_{m}(\alpha)\|)\chi_{[0,\frac{s}{N}]}(\|f_{n}(\alpha)-f_{p}(\alpha)\|)\, d\alpha.\nonumber
\end{align} Importantly $f_p-f_m$ and $f_n - f_p$ are polynomials of different degrees. As a consequence of this, the right hand side of \eqref{bound3} is in a form where the arguments used in the proof of Lemma \ref{lemma1} can be applied. In particular one can define appropriate analogues of the intervals $I_{M,q,p}$ and the points $c_{M,q,p}$. Then by analogous arguments to those given in the proof of Lemma \ref{lemma1}, it can be shown that 
\begin{equation}
\label{bound4}\sum_{\substack{1\leq m<p< n\leq N}}\int_{a}^{b}\chi_{[0,\frac{2s}{N}]}(\|f_{p}(\alpha)-f_{m}(\alpha)\|)\chi_{[0,\frac{s}{N}]}(\|f_{n}(\alpha)-f_{p}(\alpha)\|)\, d\alpha=\O_{s,a,b}\left(N\right).
\end{equation}Substituting \eqref{bound4} into \eqref{bound3} we obtain
\begin{equation*}
\frac{8}{N^2}\sum_{\substack{1\leq m<p< n\leq N}}\int_{a}^{b}\chi_{[0,\frac{s}{N}]}(\|f_{n}(\alpha)-f_{m}(\alpha)\|)\chi_{[0,\frac{s}{N}]}(\|f_{n}(\alpha)-f_{p}(\alpha)\|)\, d\alpha=\O_{s,a,b}\left(\frac{1}{N}\right).
\end{equation*}

%Defining appropriate analogues of $I_{M,l,j}$ and $c_{M,l,j}$ we obtain:
%\begin{align*}
%&\sum_{1\leq j< l\leq k\leq N}\int_{a}^{b}\chi_{s,N}(f_{l}(\alpha)-f_{j}(\alpha))\chi_{s,N}(f_{k}(\alpha)-f_{l}(\alpha))\\
%&\leq \sum_{1\leq j< l\leq k\leq N}\sum_{M=\lfloor f_{l}(a)-f_{j}(a)\rfloor}^{\lceil f_{l}(b)-f_{j}(\b)\rceil}\int_{I_{M,l,j}}\chi_{s,N}(f_{k}(\alpha)-f_{l}(\alpha))\\
%&\leq  \sum_{1\leq j< l\leq k\leq N}\sum_{M=\lfloor f_{l}(a)-f_{j}(a)\rfloor}^{\lceil f_{l}(b)-f_{j}(\b)\rceil} \frac{\L(I_{M,l,j})2s/N}{1-2s/N}+\O(\frac{2s}{N(f_k-f_l)'(c_{M,l,j})}\\
%&\leq \frac{2s}{N-2s} \sum_{1\leq j< l\leq k\leq N}\frac{(b-a)2s/N}{1-2s/N}+\O(\frac{2s}{Na^{x_k}})\\
%&+\frac{2s}{N}\sum_{1\leq j< l\leq k\leq N}\sum_{M=\lfloor f_{l}(a)-f_{j}(a)\rfloor}^{\lceil f_{l}(b)-f_{j}(\b)\rceil}(M-1)^{-x_k/x_l}
%\end{align*}
%Importantly $$\frac{2s}{N-2s} \sum_{1\leq j< l\leq k\leq N}\frac{(b-a)2s/N}{1-2s/N}+\O(\frac{2s}{Na^{x_k}})=\O(N)$$. Similarly, by the previous arguments
%\begin{align*}
%\sum_{1\leq j< l\leq k\leq N}\sum_{M=\lfloor f_{l}(a)-f_{j}(a)\rfloor}^{\lceil f_{l}(b)-f_{j}(\b)\rceil}(M-1)^{-x_k/x_l}&\leq \sum_{1\leq j< l\leq k\leq N}a''^{-k}.
%\end{align*}And importantly $$ \sum_{1\leq j< l\leq k\leq %N}a''^{-k}=\O(\frac{2s}{N}).$$

\end{proof}
Substituting the bounds provided by Lemma \ref{lemma1} and Lemma \ref{lemma2} into \eqref{double split}, we see that under the hypothesis of Theorem \ref{Main thm}, the following holds for the first term in \eqref{split}
\begin{equation}
\label{first term bound}\frac{1}{N^2}\sum_{\substack{1\leq m\neq n\leq N\\ 1\leq p\neq q\leq N\\(m,n)\neq (p,q)}}\int_{a}^{b}\chi_{[0,\frac{s}{N}]}(\|f_{n}(\alpha)-f_{m}(\alpha)\|)\chi_{[0,\frac{s}{N}]}(\|f_{q}(\alpha)-f_{p}(\alpha)\|)\, d\alpha\leq 4s^2(b-a)+\O_{s,a,b}\left(\frac{1}{N}\right).
\end{equation}

\begin{lemma}
	\label{lemmab}
	Suppose $(f_n)_{n=1}^{\infty}$ is a sequence of polynomials satisfying the hypothesis of Theorem \ref{Main thm}. Then for any $[a,b]\subset(1,\infty)$ and $s>0$ we have $$\frac{1}{N^2}\sum_{1\leq m\neq n\leq N}\int_{a}^{b}\chi_{[0,\frac{s}{N}]}(\|f_{n}(\alpha)-f_{m}(\alpha)\|)\, d\alpha=\O_{s,a,b}\left(\frac{1}{N}\right)$$ and $$\frac{4s}{N}\sum_{1\leq m\neq n\leq N}\int_{a}^{b}\chi_{[0,\frac{s}{N}]}(\|f_{n}(\alpha)-f_{m}(\alpha)\|)\, d\alpha\geq 8s^2(b-a)+\O_{s,a,b}\left(\frac{1}{N}\right).$$
\end{lemma}
\begin{proof}
	To prove our result it suffices to show that 
	\begin{equation}
	\label{upper bound}
	\sum_{1\leq m\neq n\leq N}\int_{a}^{b}\chi_{[0,\frac{s}{N}]}(\|f_{n}(\alpha)-f_{m}(\alpha)\|)\, d\alpha=\O_{s,a,b}(N)
	\end{equation}
	 and 
	 \begin{equation}
	 \label{lower bound}\sum_{1\leq m\neq n\leq N}\int_{a}^{b}\chi_{[0,\frac{s}{N}]}(\|f_{n}(\alpha)-f_{m}(\alpha)\|)\, d\alpha\geq 2s(b-a)N+\O_{s,a,b}(1).
	 \end{equation}
We start by proving \eqref{upper bound}. Applying Lemma \ref{convexity lemma} together with our first and third assumptions, we see that the following holds:
\begin{align*}
	\sum_{1\leq m\neq n\leq N}\int_{a}^{b}\chi_{[0,\frac{s}{N}]}(\|f_{n}(\alpha)-f_{m}(\alpha)\|)\, d\alpha&=2\sum_{1\leq m< n\leq N}\int_{a}^{b}\chi_{[0,\frac{s}{N}]}(\|f_{n}(\alpha)-f_{m}(\alpha)\|)\, d\alpha\\ 
	&\leq 2\sum_{1\leq m< n\leq N}\left(\frac{(b-a)2s/N}{1-2s/N}+\O_{s,a,b}\left(\frac{1}{Nd_na^{d_n}}\right)\right)\\
	&= 2\sum_{1\leq m< n\leq N}\frac{(b-a)2s/N}{1-2s/N}+\O_{s,a,b}\left(\sum_{1\leq m< n\leq N}\frac{1}{Nd_na^{d_n}}\right)\\
	&\leq \frac{4s(b-a)}{N-2s}\sum_{1\leq m< n\leq N}1+ \O_{s,a,b}\left(\frac{1}{N}\right)\\
	&= \frac{4s(b-a)}{N-2s}\left(\frac{N^2}{2}+\O(N)\right)+\O_{s,a,b}\left(\frac{1}{N}\right)\\
	&=\O_{s,a,b}(N).
\end{align*}By a similar argument, this time using the lower bound from Lemma \ref{convexity lemma}, it can be shown that
\begin{align*}\sum_{1\leq m\neq n\leq N}\int_{a}^{b}\chi_{[0,\frac{s}{N}]}(\|f_{n}(\alpha)-f_{m}(\alpha)\|)\geq 2s(b-a)N+\O_{s,a,b}(1).
\end{align*}

\end{proof}
 
\begin{proof}[Proof of Proposition \ref{error prop}]
Proposition \ref{error prop} follows by substituting the bounds provided by \eqref{first term bound} and Lemma \ref{lemmab} into \eqref{split}.
\end{proof}

Equipped with Proposition \ref{error prop} we are now in a position to prove Theorem \ref{Main thm}.

\begin{proof}[Proof of Theorem \ref{Main thm}]
Let us start by fixing $[a,b]\subset (1,\infty)$ and let $s>0$ be arbitrary. By Proposition \ref{error prop} we know that $$\int_{a}^{b}\left(\frac{\#\{1\leq m\neq n\leq N^2: \|f_{n}(\alpha)-f_{m}(\alpha)\|\leq \frac{s}{N^2}\}}{N^2}-2s\right)^2 d\alpha=\O_{s,a,b}\left(\frac{1}{N^2}\right).$$ Applying Markov's inequality, we have
$$\L\left(\alpha\in[a,b]:\left(\frac{\#\{1\leq m\neq n\leq N^2: \|f_{n}(\alpha)-f_{m}(\alpha)\|\leq \frac{s}{N^2}\}}{N^2}-2s\right)^2>N^{-1/2}\right)= \O_{s,a,b}\left(\frac{1}{N^{3/2}}\right).$$ Importantly $$\sum_{N=1}^{\infty}\frac{1}{ N^{3/2}}<\infty.$$ Therefore, by the Borel-Cantelli lemma, for Lebesgue almost every $\alpha\in[a,b],$ the inequality $$\left(\frac{\#\{1\leq m\neq n\leq N^2: \|f_{n}(\alpha)-f_{m}(\alpha)\|\leq \frac{s}{N^2}\}}{N^2}-2s\right)^2>N^{-1/2}$$ is satisfied for at most finitely many values of $N$. This implies that for Lebesgue almost every $\alpha\in[a,b]$
we have $$\lim_{N\to\infty}\frac{\#\{1\leq m\neq n\leq N^2: \|f_{n}(\alpha)-f_{m}(\alpha)\|\leq \frac{s}{N^2}\}}{N^2}=2s.$$ The parameter $s$ was arbitrary. Therefore by considering a countable dense set of $s$, and applying an approximation argument, it can be shown that for Lebesgue almost every $\alpha\in[a,b]$, for any $s>0$ we have \begin{equation}
\label{square convergence}
\lim_{N\to\infty}\frac{\#\{1\leq m\neq n\leq N^2: \|f_{n}(\alpha)-f_{m}(\alpha)\|\leq \frac{s}{N^2}\}}{N^2}=2s.
\end{equation} 

To each $N\in\mathbb{N}$ we define $M_N\in\mathbb{N}$ to be the unique integer satisfying the inequalities $$M_{N}^2\leq N< (M_N +1)^2.$$ Let $s>0$ and $\epsilon>0$ be arbitrary. Since \eqref{square convergence} holds for Lebesgue almost every $\alpha\in[a,b]$ for any $s>0$, we have
\begin{align*}
&\limsup_{N\to\infty}\frac{\#\{1\leq m\neq n\leq N: \|f_{n}(\alpha)-f_{m}(\alpha)\|\leq \frac{s}{N}\}}{N}\\
\leq &\limsup_{N\to\infty}\frac{\#\{1\leq m\neq n\leq (M_N+1)^2: \|f_{n}(\alpha)-f_{m}(\alpha)\|\leq \frac{s+\epsilon}{(M_N +1)^2}\}}{M_N^2}\\
=&\limsup_{N\to\infty}\left(\frac{(M_N+1)^2}{M_N^2}\right)\frac{\#\{1\leq m\neq n\leq (M_N+1)^2: \|f_{n}(\alpha)-f_{m}(\alpha)\|\leq \frac{s+\epsilon}{(M_N +1)^2}\}}{(M_N+1)^2}\\
=&2(s+\epsilon)
\end{align*}for Lebesgue almost every $\alpha\in[a,b]$. Similarly it can be shown that for Lebesgue almost every $\alpha\in[a,b],$ we have $$\liminf_{N\to\infty}\frac{\#\{1\leq m\neq n\leq N: \|f_{n}(\alpha)-f_{m}(\alpha)\|\leq \frac{s}{N}\}}{N}\geq 2(s-\epsilon).$$ Since $s$ and $\epsilon$ were arbitrary, we may conclude that for Lebesgue almost every $\alpha\in[a,b]$ we have 
$$\lim_{N\to\infty}\frac{\#\{1\leq m\neq n\leq N: \|f_{n}(\alpha)-f_{m}(\alpha)\|\leq \frac{s}{N}\}}{N}= 2s$$ for any $s>0$. Since the interval $[a,b]$ was arbitrary this completes our proof.

\end{proof}

\noindent \textbf{Acknowledgements.} The author was supported by EPSRC grant EP/M001903/1.

\end{document}